\documentclass[11pt]{article}

 \usepackage{amsmath,amssymb,amscd,amsthm,esint}

\usepackage{graphics,amsthm,mathrsfs,amsfonts}

\usepackage{sidecap}

\usepackage{float}

\usepackage{extarrows}

\usepackage{booktabs}

\usepackage{verbatim}

\usepackage[usenames,dvipsnames]{xcolor}

\usepackage{hyperref}

\newtheorem{thm}{Theorem}[section]
\newtheorem{lemma}[thm]{Lemma}

\theoremstyle{definition}
\newtheorem{remark}[thm]{Remark}

\def\XXint#1#2#3{{\setbox0=\hbox{$#1{#2#3}{\int}$}
         \vcenter{\hbox{$#2#3$}}\kern-.5\wd0}}

\def\R{\mathbb{R}}
\def\C{\mathbb{C}}
\def\e{\varepsilon}

\def\A{\mathbf{A}}
\def\B{\mathbf{B}}

\def\b{\mathcal{B}}
\def\e{\varepsilon}

\numberwithin{equation}{section}

\begin{document}

\title{Exponential Decays of Steklov Eigenfunctions\\  for the Magnetic Laplacian}

 \author{
Zhongwei Shen  \thanks{\noindent Institute for Theoretical Sciences, Westlake University,
No. 600 Dunyu Road, Xihu District, Hangzhou, Zhejiang  310030, P.R. China.\\
\emph{E-mail address}: \texttt{shenzhongwei@westlake.edu.cn} } 
 }
\date{}
\maketitle

\begin{abstract} 

Consider the Dirichlet-to-Neumann map $\Lambda_\beta$
associated with the  Schr\"odinger operator $(D+\beta \A)^2$ with a magnetic potential in a bounded Lipschitz 
domain $\Omega$,
where $\beta>1$ is the field strength parameter.
Assume that the magnetic field $\B=\nabla \times \A$ is of finite type.
We show that if $\beta>\beta_0$,
 the ground state for $\Lambda_\beta$ decays exponentially away from a neighborhood of the  subset
of  $\partial\Omega$, on which $\B$ vanishes to the maximal order.

\medskip


\medskip

\noindent {\it MR (2020) Subject Classification}: 35P20; 35Q40.

\end{abstract}


\section{Introduction}\label{section-1}

Let  $D=-i \nabla$ and $\A=(A_1, A_2, \dots, A_d)\in C^1(\overline{\Omega}; \R^d)$,
where $\Omega$ is a bounded Lipschitz domain in $\R^d$, $d\ge 2$.
Consider the Schr\"odinger operator with a magnetic potential, 
\begin{equation}\label{op-1}
H (\beta \A) =(D+\beta \A)^2, 
\end{equation}
also called the magnetic Laplacian, 
where $\beta> 1 $ is the field strength  parameter. 
The asymptotic behaviors of eigenvalues and eigenfunctions for the operator \eqref{op-1} with either the
Dirichlet condition $u=0$ on $\partial \Omega$ or the Neumann condition  $n \cdot (D +\beta \A)u=0$ on $\partial \Omega$, 
as $\beta \to \infty$, have been studied extensively since 1980's.
 The asymptotic expansions of the ground state  energies, with remainder estimates,
as well as exponential localizations of the ground states, have been established 
under various assumptions. 
For a survey and  references  in this area as well as applications to  the theory of  superconductivity, 
we  refer the reader to two expository books, \cite{Helffer-book} by S. Fournais and B. Helffer and \cite {Raymond-book}  by N. Raymond.
Let $\B=\nabla \times \A$ denote the magnetic field.
 It is  known that the leading orders of the ground state energies 
depend on the highest  vanishing order of  $\B$ on  $\overline{\Omega}$  \cite{Montgomery, Helffer-1996, Lu-1999, Helffer-2001, Pan-2002, Miqueu-2018,
Dauge-2020, Shen-2025}.
Moreover, the ground states decay exponentially away from the  subset of $\overline{\Omega}$, on which 
$\B$ vanishes to the highest order \cite{Helffer-1996, Miqueu-2018, Dauge-2020}.

In recent years, there is a growing interest in the study of spectrum properties of
the Dirichlet-to-Neumann map, associated with the Laplacian or a general Schr\"odinger operator.
For references in the case of the Laplacian on a compact manifold, we refer the reader to a survey article \cite{C-2023}.
For  recent works in the case of the magnetic Laplacian, see    \cite{CGHP, Helffer-F2024, Helffer-2024, Shen-2025} and their references.
Let $n$ denote the outward unit normal to $\partial\Omega$.
The Dirichlet-to-Neumann map  associated with $H(\beta \A)$ is defined by
\begin{equation}\label{DN}
\Lambda_\beta : f  \to  n \cdot i(D+\beta \A) u,
\end{equation}
where $f\in H^{1/2}(\partial\Omega; \C)$ and $u\in H^1(\Omega; \C)$ is the solution of the 
Dirichlet problem, 
\begin{equation}\label{EF}
\left\{
\aligned
(D+\beta \A)^2 u & =0 & \quad & \text{ in } \Omega, \\
u&=f & \quad & \text{ on } \partial\Omega.
\endaligned
\right.
\end{equation}
With its quadratic-form realization on $L^2(\partial\Omega)$, $\Lambda_\beta$ is a nonnegative self-adjoint operator with
compact resolvent. Consequently, its spectrum is real and discrete, every eigenvalue has finite
multiplicity, and the eigenvalues tend to $\infty$; see \cite[Example 2.2 and Section 7]{Elst}.
Let 
$\lambda^{D\!N}(\beta \A, \Omega)$ denote  the ground state energy for $\Lambda_\beta$.
Under the assumption that $\B$ 
 is of finite type on $\partial\Omega$, it was proved by the present author in \cite{Shen-2025} that
\begin{equation}\label{up-low}
c \beta^{\frac{1}{\kappa_0 +2}}
\le \lambda^{D\!N} (\beta \A, \Omega)
\le C  \beta^{\frac{1}{\kappa_0 +2}}
\end{equation}
for $\beta >\beta_0$,
where $\kappa_0\ge 0$ is the highest vanishing order of $|\B|$ on $\partial\Omega$ (see \eqref{kappa}).
We also obtained a one-term asymptotic expansion, with reminder estimates, for 
$\lambda^{D\!N} (\beta \A, \Omega)$ under more restrictive  conditions.
The purpose of this paper is to study the exponential decays of eigenfunctions for the operator $\Lambda_\beta$. 

 We shall assume that the magnetic field $\B$ is of finite type on $\overline{\Omega}$
 throughout this paper.
This means that there exist  some integer $\kappa_*\ge 0$ and $\tau_*>0$ such that
\begin{equation}\label{cod-1}
\sum_{|\alpha|\le \kappa_*} |\partial^\alpha \B (x)| \ge \tau_*
\end{equation}
for any $x\in \overline{\Omega}$.
For  $t>0$, let 
\begin{equation}\label{E-1}
W_\beta (t)  =\big\{ x\in \partial\Omega:   \widetilde{m} (x, \beta \B) \le t \big\}
\end{equation}
be a magnetic well on the boundary $\partial\Omega$, where
\begin{equation}\label{mm}
\widetilde{m}(x, \beta \B)
=\sum_{|\alpha|\le \kappa_*} \beta ^{\frac{1}{|\alpha|+2}}
 |\partial^\alpha \B (x)|^{\frac{1}{|\alpha|+2}}.
 \end{equation}
 For $x\in \overline{\Omega}$ and $E\subset \overline{\Omega}$, the distance $d_\beta (x, E)$ is defined by
 \begin{equation*}
 d_\beta (x, E) =\inf\big \{ d_\beta (x, y): \ y \in E\big\},
 \end{equation*}
where $d_\beta (x, y)$ is an Agmon distance with respect to the metric $\widetilde{m}(x, \beta \B) ds^2$, given by
\begin{equation}\label{dist}
d_\beta (x, y)
=\inf \left\{
\int_0^1 \widetilde{m}(\gamma(t), \beta \B) |\gamma^\prime (t)|\, dt :\ \ \
\gamma: [0, 1] \to \overline{\Omega}, \gamma (0)=x \text{ and } \gamma (1)=y \right\}.
\end{equation}

The following is the main result of the paper.

\begin{thm}\label{main-thm}
Let $\Omega$ be a bounded Lipschitz domain in $\R^d$, $d\ge 2$ and $\A\in C^\infty(\overline{\Omega}, \R^d)$.
Assume that $\B$ is of finite type on $\overline{\Omega}$.
Let $u\in H^1(\Omega; \C)$ be a weak solution of
\begin{equation}\label{eigen} 
\left\{
\aligned
(D+\beta \A)^2 u &=0 & \quad & \text{ in } \Omega,\\
n \cdot i (D+\beta \A) u &= \lambda u & \quad & \text{ on } \partial\Omega
\endaligned
\right.
\end{equation}
for some $\lambda>0$.
Then, if $\beta > \beta_0$, 
\begin{equation}\label{main-1}
|u(x)|\le C \lambda^{\frac{d-1}{2}}
\exp\big\{- \e  d_\beta  (x, W_\beta (T\lambda))\big\}
 \| u \|_{L^2(\partial\Omega)}
\end{equation}
for any $x\in \overline{\Omega}$.
The constants $C>1, \beta_0>1, T>1,  \e >0$ depend only on $\Omega$ and $\B$.
\end{thm}

Under the finite-type condition \eqref{cod-1},  for each  $x\in \overline{\Omega}$,
there exists an integer $\kappa=\kappa (x)\ge 0$ such that
$\partial^\alpha \B(x)=0$ for any $\alpha$ with $|\alpha |\le \kappa-1$, and
$\partial^\alpha \B (x)\neq  0$ for some $\alpha$ with $|\alpha|=\kappa$.
Let 
\begin{equation}\label{kappa}
\kappa_0 =\max \{ \kappa (x): x\in \partial\Omega \}
\end{equation}
and
\begin{equation*}
\Gamma_0 =\big\{ x\in \partial\Omega: \kappa (x)=\kappa_0\big \}.
\end{equation*}
Note that by \eqref{up-low},  if $\lambda=\lambda^{D\!N}(\beta\A, \Omega)$, then $\lambda
\approx \beta^{\frac{1}{\kappa_0+2}}$ and
\begin{equation}\label{WW}
W_\beta (T\lambda)
\subset \bigg\{ x\in \partial\Omega: 
\sum_{|\alpha|\le \kappa_0-1} 
\beta^{\frac{1}{|\alpha|+2}}
|\partial^\alpha \B(x)|^{\frac{1}{|\alpha|+2}}
\le C T \beta^{\frac{1}{\kappa_0+2}} \bigg\}
\end{equation}
for $\beta > \beta_0$.
It follows by Theorem \ref{main-thm}
that the eigenfunctions associated with $\lambda^{D\!N} (\beta \A, \Omega)$
are exponentially localized in a neighborhood of $\Gamma_0$.

We briefly describe our approach to Theorem \ref{main-thm} and the organization of the paper. 
In Section 2 we study the Agmon distance \eqref{dist} and show that 
\begin{equation}\label{poly}
\widetilde{m}(x, \beta \B)
\le C \widetilde{m}(y, \beta \B) \big\{ 1+ d_\beta (x, y) \big\}^{\ell_1}
\end{equation}
for any $x, y\in \overline{\Omega}$, if $\beta $ is sufficiently large.
The estimate \eqref{poly} allows us to control the variation of the local magnetic scale along the Agmon metric.
In Section 3 we establish an $L^2$ bound for solutions of \eqref{eigen}, using the inequality 
\begin{equation}\label{key}
\int_{\partial\Omega}
\widetilde{m} (x, \beta \B) |\psi|^2\, dx
+
\int_\Omega \widetilde{m} (x, \beta \B)^2  |\psi|^2\, dx
\le C \int_\Omega |(D+\beta \A)\psi|^2\, dx
\end{equation}
for any $\psi \in H^1(\Omega; \C)$.
This is done by regularizing the distance from $W_\beta (T\lambda)$ and testing the equation with the corresponding 
exponential weight.
The coercive estimate in \eqref{key}, which was  proved in \cite{Shen-2025},
is used to absorb the terms created by differentiating the weight.
The weighted $L^2$ bound is upgraded to the  pointwise bound  \eqref{main-1}
 in Section 4. The proof relies on a boundary $L^\infty$ estimate based on  the Neumann function $N_\A(x, y)$
for the operator $(D+\A)^2$ in $\Omega$, 
constructed  in \cite{Shen-2025-2}. 
Finally, we give two examples in Section 5.

\medskip

\noindent{\bf Acknowledgement.}
The author thanks the anonymous referee for  helpful comments that improved the quality of the manuscript.


\section{An Agmon distance function}

Let $\Omega$ be a bounded Lipschitz domain in $\R^d$. 
For $t>0$, let
 \begin{equation}\label{O}
 \mathcal{O}_t = \{ x\in \R^d: \text{\rm dist}(x, \Omega)< t \}.
 \end{equation}
Let $\A\in C^\infty(\overline{\Omega}; \R^d)$. 
We extend $\B=\nabla \times \A$ to $\R^d$ so that $\|\B \|_{C^{\kappa_*+1}(\R^d)}\le C \| \B \|_{C^{\kappa_*+1}(\overline{\Omega})}$
 (there is no need to extend $\A$).
Suppose   $\B$ satisfies the condition \eqref{cod-1} for any $x\in \overline{\Omega}$.
By compactness, there exists $r_0>0$, depending on
 $\Omega$, $(\kappa_*, \tau_*)$ in \eqref{cod-1} and $\| \B\|_{C^{\kappa_*+1}(\overline{\Omega})}$, such that
 \begin{equation}\label{cod-2}
\sum_{|\alpha|\le \kappa_*} |\partial^\alpha \B (x)| \ge (\tau_*/2)
\end{equation}
  for any
 $x\in \mathcal{O}_{4r_0}$.
 Let $\b(x, r)$ denote the ball centered at $x$ of radius $r$.
 We assume that  $r_0$ is sufficiently small so that for any $x_0\in \partial\Omega$,
 $\b(x_0, 4r_0)\cap \Omega$ is given by the region above a Lipschitz graph in a coordinate system,
 obtained from the standard one through rotation and translation.
  
 \begin{lemma}\label{lemma-01}
 Let $\b(x, r) \subset \mathcal{O}_{3r_0}$. Then 
 \begin{equation}\label{B-0}
 \sup_{\b(x, r)} |\partial^\alpha \B|\le 
 \frac{C}{r^{|\alpha|}} \fint_{\b(x, r)} |\B|
 \end{equation}
 for any $\alpha$ with $|\alpha|\le \kappa_*$,  and 
 \begin{equation}\label{B-1a}
 \sup_{\b(x, r)} |\B |\ge c\, r^{\kappa_*},
 \end{equation}
where  $C, c>0$ depend only on $\Omega$, $(\kappa_*, \tau_*)$ in \eqref{cod-1} and $\| \B\|_{C^{\kappa_*+1}(\overline{\Omega})}$.
 \end{lemma}

\begin{proof}

See \cite[Lemma 10.1]{Shen-2025-2}.
\end{proof}

Let $\beta_0=2 c^{-1} r_0^{-(\kappa+2)}$, where $c$ is given by \eqref{B-1a}.
For any $x\in \R^d$ and $\beta> \beta_0$, define $m(x, \beta\B)$ by
\begin{equation}\label{m}
\frac{1}{m(x, \beta \B)}
=\sup \Big\{ r> 0 : \sup_{\b(x, r)} |\beta \B| \le \frac{1}{r^2} \Big\}.
\end{equation}
Motivated by Fefferman-Phong's uncertainty principle \cite{Fefferman-1983}, the function $m(x, V)$ was introduced
by the present author in \cite{Shen-1995} for a nonnegative function $V$ in a reverse H\"older class.
 It follows from \eqref{B-1a} that  if $\beta>\beta_0$, then
$m(x, \beta\B)^{-1}<  r_0$ for any $x\in \mathcal{O}_{2r_0}$.
Moreover, by \cite[Lemma 10.2]{Shen-2025-2}, 
\begin{equation}\label{m-eq}
c\,  \widetilde{m}(x, \beta \B)
\le m (x, \beta \B)\le C \widetilde{m} (x, \beta \B)
\end{equation}
for any $x\in \Omega$, where $\widetilde{m}(x, \beta \B)$ is given by \eqref{mm} and 
$C, c>0$ depend only on $\Omega$, $(\kappa_*, \tau_*)$ in \eqref{cod-1} and
$\| \B\|_{C^{\kappa_* +1}(\overline{\Omega})}$.
The estimates in \eqref{m-eq} allow us to interchange $m(x,\beta \B)$ freely with 
$\widetilde{m}(x, \beta \B)$.
To see \eqref{m-eq},  we note that if the first nonzero derivative of $\B$ at $x$ has order $k$, Taylor expansion 
suggests $\max |\B|\approx |\partial^k \B (x)| r^k$ in the ball $\b(x, r)$.
Balancing $\beta |\partial^k \B (x)| r^k$ with $r^{-2}$ gives
$$
m(x, \beta \B ) \approx \beta^{\frac{1}{k+2}} |\partial^k \B (x)|^{\frac{1}{k+2}}.
$$
Taking all $k \le \kappa_*$ into account explains the expression $\widetilde{m}(x, \beta \B)$ in \eqref{mm} and 
the comparison in \eqref{m-eq}.

\begin{lemma}\label{lemma-02}
Let $\b(x_0, r_0)\subset \mathcal{O}_{2r_0}$. 
Then, for $\beta> \beta_0$, 
\begin{align}
m(x, \beta \B) & \le C \{ 1+ |x-y| m(y, \beta \B) \}^{\ell_0} m(y, \beta \B)\label{m-a}, \\
m(y, \beta \B)& \ge \frac{c \, m(x, \beta \B)}{ \{ 1+ |x-y| m(x, \beta \B)\}^{\frac{\ell_0}{\ell_0+1}}}\label{m-b},
\end{align}
for any $x, y\in \b(x_0, r_0)$,
where $C, c, \ell_0>0$ depend only on $\Omega$, $(\kappa_*, \tau_*)$ in \eqref{cod-1} and $\| \B\|_{C^{\kappa_*+1}(\overline{\Omega})}$.
\end{lemma}

\begin{proof}
Recall that a nonnegative function $w$ is a $B_\infty$ weight on an open set $U$ if
$$
\sup_{\b} w \le C \fint_{\b} w
$$
for every ball $\b\subset U$.
It follows from \eqref{B-0} with $\alpha =0$ that 
$$
\sup_{\b} |\beta \B|\le C \fint_\b |\beta \B|
$$
for any ball  $\b\subset \mathcal{O}_{3r_0}$. Thus, 
$|\beta \B|$ is a $B_\infty$ weight on the set $\mathcal{O}_{3r_0}$.
As a result, the estimates \eqref{m-a}-\eqref{m-b} for  $x, y\in \b(x_0, r_0) \subset \mathcal{O}_{2r_0}$
follow from the proof of \cite[Lemma 1.4]{Shen-1995}.
The condition $m(x, \beta\B)^{-1} < r_0$ is needed to carry out the argument in a fixed ball $\b(x_0, r_0)$ instead of $\R^d$.
\end{proof}

Let $d_\beta (x, y)$ be defined by \eqref{dist}.


\begin{thm}\label{lemma-a1}
Suppose $\B$ satisfies the finite-type condition \eqref{cod-1}.
There exist $\beta_0>0$, $\ell_1>0$ and $C>0$, depending only on $\Omega$, $(\kappa_*, \tau_*)$ in \eqref{cod-1}  
and $\|\B \|_{C^{\kappa_*+1}(\overline{\Omega})}$,
such that
\begin{equation}\label{a-0}
m (x, \beta \B) \le C\,  m(y, \beta \B) \big\{ 1+ d_\beta (x, y) \big\}^{\ell_1}
\end{equation}
for any $x, y\in \overline{\Omega}$ and $\beta> \beta_0$.
\end{thm}

\begin{proof}

Since $m(x, \beta \B)$ satisfies \eqref{m-a}-\eqref{m-b} for any $x, y\in \b(x_0, r_0)\subset \mathcal{O}_{2r_0}$,
it follows from   \eqref{m-eq} and  \cite[Lemma 4.18]{Shen-1996} that 
the estimate \eqref{a-0} holds for any $x, y\in \overline{\Omega}$
with $|x-y|< r_0$.

To deal with the general case, we cover $\overline{\Omega}$ by a finite number of balls
$\{\b(y_j, r_0/4), j=1, 2, \dots, N\}$, where $y_j\in \overline{\Omega}$ and $N$ depends on $\Omega$ and $r_0$.
Fix $x, y\in \overline{\Omega}$.
Choose a piece-wise $C^1$ curve $\gamma: [0, 1] \to \overline{\Omega}$ such that $\gamma (0)=x$, $\gamma (1)=y$, and 
$$
\int_0^1 \widetilde{m}(\gamma(t), \beta \B) |\gamma^\prime (t)| dt \le   d_\beta (x, y) +1.
$$
Let $\Gamma =\{ \gamma(t): t \in [0, 1] \}$.
Without the loss of generality, we may assume that 
$$
\Gamma \subset \cup_{j=1}^{N_0} \b(y_j, r_0/4),\quad 
\Gamma\cap \b(y_j, r_0/4)\neq \emptyset,
$$
$x\in \b(y_1, r_0/4)$, $y\in \b(y_{N_0}, r_0/4)$,
and $\b(y_j, r_0/4) \cap \b(y_{j+1}, r_0/4) \neq \emptyset$, where $N_0\le N$.
Let $z_j \in \b(y_j, r_0/4) \cap \Gamma$ for $j=2, \dots, N_0-1$, and $z_1=x$, $z_{N_0}=y$.
Since $|z_j -z_{j+1}|<   r_0$, we have 
\begin{equation*}
\aligned
m(z_j , \beta \B) & \le C_0 m(z_{j+1} , \beta \B) \big\{ 1+ d_\beta(z_j , z_{j+1})\big\}^{\ell_0}\\
& \le C_0 m(z_{j+1} , \beta \B)\big \{ 2+ d_\beta(x, y)\big\}^{\ell_0}
\endaligned
\end{equation*}
for $j=1, \dots, N_0-1$.
It follows that
\begin{equation*}
m(x, \beta \B)
\le C_0^{N_0} m(y, \beta \B) \{ 2+ d_\beta(x, y) \}^{\ell_0 N_0}.
\end{equation*}
Since $N_0\le N$, this gives \eqref{a-0} with $\ell_1=N\ell_0$ for any $x, y\in \overline{\Omega}$.
\end{proof}


\section{$L^2$ bounds}

For $\lambda>0$ and $\beta>\beta_0$, define
\begin{equation}\label{E}
W_\beta (\lambda)=\big\{ x\in \partial\Omega: \ \widetilde{m}(x, \beta\B)< \lambda \big\}
\end{equation}
and
\begin{equation}\label{d}
d_{\beta, \lambda} (x)=d_\beta (x, W_\beta (\lambda))
=\inf \big\{ d_\beta (x, y): y \in W_\beta (\lambda) \big\},
\end{equation}
where $d_\beta (x, y)$ is defined by \eqref{dist}.
It follows from \eqref{a-0} and \eqref{m-eq} that for $\beta>\beta_0$, 
\begin{equation}\label{a-3}
\widetilde{m}(x, \beta \B)
\le C \lambda \{ 1+ d_{\beta, \lambda} (x) \}^{\ell_1}
\end{equation}
for any $x\in \overline{\Omega}$.

We need to regularize the function $d_{\beta, \lambda} (x)$.
 Using the properties \eqref{m-a}-\eqref{m-b}, we may find a finite number of balls 
$\{ \b(z_j,  t_j): j=1, 2, \dots, J \}$ with $z_j \in \overline{\Omega} $ and $t_j =(1/2){m}(z_j, \beta\B)^{-1}$ such that 
$$
1\le \sum_j \chi_{\b(z_j, t_j)}
\le \sum_j \chi_{\b(z_j, 2t_j)} \le C \quad \text{ in } \overline{\Omega},
$$
where $C$ depends only on $\Omega$ and $(\kappa_*, \tau_*)$ in \eqref{cod-1}.
The choice of $\beta_0$ gives $m(z_j, \beta \B)^{-1}< r_0$, and
hence $0< t_j < r_0/2$.
This allows us to construct a partition  of unity  $\{ \varphi_j\}$ for $\overline{\Omega}$ such that 
$\varphi_j \in C_0^\infty(\b(z_j, 2t_j); \R)$, $\sum_j \varphi_j =1$ in $\overline{\Omega}$, and 
$$
0\le \varphi_j \le 1, \quad
|\nabla \varphi_j |\le C t_j^{-1},\quad
|\nabla^2 \varphi_j|\le C t_j^{-2}.
$$
Define
\begin{equation}\label{psi}
\psi_{\beta, \lambda} (x)
= \sum_j  d_{\beta, \lambda} (z_j ) \varphi_j (x).
\end{equation}
Clearly, $\psi_{\beta, \lambda} \in C_0^\infty(\R^d; \R)$.

\begin{lemma}\label{lemma-a2}
Let $\beta> \beta_0$ and  $\psi_{\beta, \lambda}$ be defined by \eqref{psi}. Then

\begin{enumerate}

\item

For $x\in \overline{\Omega}$, $|\psi_{\beta, \lambda}(x)-  d_{\beta, \lambda} (x)|\le C $.

\item

For $x\in \overline{\Omega}$, we have 
\begin{equation}\label{psi-1}
|\nabla \psi_{\beta, \lambda}  (x)|\le C m(x, \beta \B) \quad \text{ and } \quad
|\nabla^2 \psi _{\beta, \lambda}(x)|\le C m(x, \beta \B)^2.
\end{equation}

\end{enumerate}
\end{lemma}

\begin{proof}

To see (1),  note that for $x\in \overline{\Omega}$, 
$$
\psi_{\beta, \lambda} (x) - d_{\beta, \lambda} (x)=\sum_j ( d_{\beta, \lambda} (z_j)-d_{\beta, \lambda} (x)) \varphi_j (x),
$$
where we have used the fact $\small\sum_j \varphi_j=1$ on $\overline{\Omega}$.
It follows that
$$
|\psi_{\beta, \lambda} (x) - d_{\beta, \lambda} (x)|\le
\sum_{j} d_\beta   (z_j , x) \varphi_j (x).
$$
Since $d_\beta (x, y)\le C$ if $|x-y|< C{m}(x, \beta \B)^{-1}$, this yields  $|\psi_{\beta, \lambda} (x)-d_{\beta, \lambda} (x) |\le C$.

To prove (2), we use the observation that
$$
\nabla \psi_{\beta, \lambda} (x)
=\sum_j d_{\beta, \lambda} (z_j) \nabla \varphi_j (x)
=\sum_j (d_{\beta, \lambda} (z_j) -d_{\beta, \lambda} (z_k) ) \nabla \varphi_j (x),
$$
where $z_k$ is chosen so that $x\in \b(z_k, t_k)$.
Thus,
$$
\aligned
|\nabla \psi_{\beta, \lambda} (x)|
 & \le \sum_j d_{\beta} (z_j, z_k) |\nabla \varphi_j (x)|\\
&\le C \sum_j |\nabla \varphi_j (x)|,
\endaligned
$$
where we have used the fact $|z_j -z_k|\le C m(z_k, \beta\B)^{-1}$ if $x\in \b(z_j, 2t_j)$.
Since $|\nabla \varphi_j (x)|\le C t_j^{-1}\le C  m(x, \beta\B)$  for $x\in \b(z_j, 2t_j)$, 
this leads to $|\nabla \psi_{\beta, \lambda} (x)|\le C m(x, \beta\B)$.
The second inequality in \eqref{psi-1} follows in the same manner.
\end{proof}

\begin{thm}\label{theorem-m}
Suppose that \eqref{cod-1} holds for any $x\in \overline{\Omega}$. 
Let $\beta> \beta_0$.
Then
\begin{align}
\int_{\partial\Omega}
m(x, \beta \B) |\psi|^2 \, dx 
 & \le C \int_\Omega |(D+\beta\A)\psi|^2\, dx\label{m-m1}, \\
 \int_{\Omega}
m(x, \beta \B)^2 |\psi|^2 \, dx 
 & \le C \int_\Omega |(D+\beta\A)\psi|^2\, dx\label{m-m2}
\end{align}
for any $\psi \in H^1(\Omega; \C)$, 
where $C$ depends on $\Omega$, $(\kappa_*, \tau_*)$ in \eqref{cod-1} and $\| \B\|_{C^{\kappa_*+1}(\overline{\Omega})}$.
\end{thm}

\begin{proof}
This  was proved in \cite[Theorem 3.8]{Shen-2025}.
\end{proof}

\begin{lemma}\label{lemma-L}
Let $u\in H^1(\Omega; \C)$ be a weak solution of the Neumann problem,
\begin{equation}\label{NP}
(D+\A)^2 u=0 \quad \text{ in } \Omega \quad \text{ and } \quad
n \cdot i (D+\A) u=g \quad \text{ on } \partial\Omega,
\end{equation}
where $g\in L^2(\partial\Omega; \C)$. Then
\begin{equation}\label{we1}
\aligned
\int_\Omega |(D+\A) (ue^{\e \psi})|^2
& \le 6\e^2 \int_\Omega  e^{2\e \psi} |u|^2 |\nabla \psi|^2
+ 2\e \int_\Omega e^{2\e \psi} |u|^2  |\Delta \psi|\\
 &\qquad
 +2 \int_{\partial\Omega}  e^{2\e \psi} |u| |g|
+2\e \int_{\partial\Omega} e^{2\e \psi} |u|^2 |\nabla \psi|, 
\endaligned
\end{equation}
where $\e>0$ and $\psi\in C^2(\R^d, \R)$.
\end{lemma}

\begin{proof}

A computation shows that
$$
(D+\A)(ue^{\e \psi})= e^{\e\psi} (D+\A) u + u De^{\e\psi}
$$ 
and
$$
(D+\A)^2 (u e^{\e\psi})
= e^{\e \psi} (D+\A)^2 u
+ 2 (D+\A) u \cdot De^{\e \psi}
+ u D^2 e^{\e \psi}.
$$
It follows that if $(D+\A)^2 u=0$ in $\Omega$, then
$$
\aligned
(D+\A)^2 (ue^{\e \psi})
&= 2 (D+\A) ( u e^{\e\psi}) \cdot \e D \psi
+ u e^{\e\psi}
\left\{ -\e^2 D\psi \cdot D\psi + \e D^2 \psi\right \}.
\endaligned
$$
Using integration by parts, we see that
$$
\aligned
\int_\Omega |(D+\A)(ue^{\e\psi})|^2
&=\int_\Omega (D+\A)^2 (ue^{\e \psi})
\cdot \overline{u} e^{\e\psi}
+i \int_{\partial \Omega}
n\cdot (D+\A) (ue^{\e \psi}) \cdot \overline{u} e^{\e \psi}\\
&=2\e \int_\Omega
(D+\A) (e^{\e\psi} u)
\cdot D \psi \cdot \overline{u} e^{\e\psi}\\
 & \quad +\int_\Omega
e^{2\e \psi} |u|^2 \left\{ -\e^2 D\psi \cdot D\psi + \e D^2 \psi\right \}\\
 & \quad + \int_{\partial\Omega} g \overline{u} e^{2\e \psi}
+ i \e \int_{\partial\Omega}
n\cdot D\psi |u|^2 e^{2\e \psi}. 
\endaligned
$$
The inequality \eqref{we1} now follows by applying the Cauchy inequality to the first term in the right-hand side of
the equality above.
Note that we have used the integration by parts in a Lipschitz domain $\Omega$.
This can be justified by approximating $\Omega$  from inside by a sequence of smooth domains 
with uniform Lipschitz characters.
\end{proof}

\begin{thm}\label{L-2-thm}
Let $\Omega$ be a bounded Lipschitz domain in $\R^d$, $d\ge 2$ and $\A\in C^{\kappa_*+1} (\overline{\Omega}; \R^d)$.
Suppose $\B$ satisfies the condition \eqref{cod-1} for any $x\in \overline{\Omega}$.
Let $u\in H^1(\Omega; \C)$ be a weak solution of 
\begin{equation}\label{EF-1}
\left\{
\aligned
(D+\beta\A )^2 u  & =0& \quad & \text{ in } \Omega, \\
n \cdot i (D+\beta \A)u & =\lambda u & \quad & \text{ on } \partial\Omega
\endaligned
\right.
\end{equation}
for some $\lambda>0$. Then, for $\beta> \beta_0$, 
\begin{equation}\label{L-2}
\int_{\partial\Omega} 
e^{2\e  d_{\beta, T\lambda} (x)  } |u (x)|^2 \, dx \le C \int_{\partial\Omega} |u|^2, 
\end{equation}
\begin{equation}\label{L-3}
\int_{\Omega} 
e^{2\e  d_{\beta, T \lambda} (x)  } m(x, \beta \B)^2  |u (x)|^2 \, dx \le C \lambda \int_{\partial\Omega} |u|^2, 
\end{equation}
\begin{equation}\label{L-3a}
\int_{\Omega} 
e^{2\e  d_{\beta, T \lambda} (x)  }   |(D+\beta \A) (x)|^2 \, dx \le C \lambda \int_{\partial\Omega} |u|^2, 
\end{equation}
where $d_{\beta, T \lambda} (x)$ is defined by \eqref{d}.
The constants $\e>0$, $\beta_0> 1$, $T>1$ and  $C>1$ depend on $\Omega$, $(\kappa_*, \tau_*)$ in \eqref{cod-1} 
and $\| \B\|_{C^{\kappa_0+1}(\overline{\Omega})}$.
\end{thm}

\begin{proof}

Let $\psi(x)=\psi_{\beta, T\lambda}(x)$ be given by \eqref{psi}, where $T>1$ is to be determined.
It follows by Theorem \ref{theorem-m} that
$$
\int_{\partial\Omega} m (x, \beta\B) |e^{\e \psi} u|^2
+\int_\Omega m(x, \beta \B)^2 |e^{\e \psi} u|^2
\le C \int_\Omega | (D+\beta \A) (e^{\e \psi} u)|^2
$$
for any $\e>0$.
This, together with \eqref{we1}, leads to 
$$
\aligned
  \int_{\partial\Omega}  & m (x, \beta\B) |e^{\e \psi} u|^2
+\int_\Omega m(x, \beta \B)^2 |e^{\e \psi} u|^2
+\int_\Omega |e^{\e \psi} (D+\beta \A)u|^2\\
&\le  C \int_\Omega | (D+\beta \A) (e^{\e \psi} u)|^2
+ C\e^2 \int_\Omega  e^{2\e \psi} |u|^2 |\nabla \psi|^2\\
&  \le C\e^2 \int_\Omega  e^{2\e \psi} |u|^2 |\nabla \psi|^2
+ C\e \int_\Omega e^{2\e \psi} |u|^2  |\Delta \psi|\\
 &\qquad
 +C\lambda   \int_{\partial\Omega}  e^{2\e \psi} |u|^2
+C \e \int_{\partial\Omega} e^{2\e \psi} |u|^2 |\nabla \psi|.
\endaligned
$$
Using $|\nabla \psi (x) |\le C m(x, \beta \B)$ and $|\nabla^2 \psi (x)|\le C m(x, \beta \B)^2$, we obtain 
$$
\aligned
  \int_{\partial\Omega}  & m (x, \beta\B) |e^{\e \psi} u|^2
  +\int_\Omega m(x, \beta \B)^2 |e^{\e \psi} u|^2
  +\int_\Omega |e^{\e \psi} (D+\beta \A)u|^2\\
& \le C (\e^2 +\e) 
\int_\Omega m(x, \beta \B)^2 |e^{\e \psi} u|^2
+C\lambda   \int_{\partial\Omega}  e^{2\e \psi} |u|^2
+C \e  \int_{\partial\Omega}  & m (x, \beta\B) |e^{\e \psi} u|^2.
\endaligned
$$
By choosing $\e>0$ sufficiently small, we deduce that
\begin{equation}\label{L-2-1}
  \int_{\partial\Omega}   m (x, \beta\B) |e^{\e \psi} u|^2
  +\int_\Omega m(x, \beta \B)^2 |e^{\e \psi} u|^2
  +\int_\Omega |e^{\e \psi} (D+\beta \A)u|^2
  \le 
C\lambda   \int_{\partial\Omega}  e^{2\e \psi} |u|^2.
\end{equation}

Finally, note that $\psi=0$ on $W_\beta (T\lambda)$ and $m(x, \beta \B)\ge T\lambda$ on $\partial\Omega\setminus W_\beta (T\lambda)$. It follows that 
$$
\aligned
\int_{\partial\Omega} e^{2\e \psi} |u|^2
& =\int_{W_\beta  (T\lambda) }e^{2\e\psi} |u|^2
+\int_{\partial\Omega\setminus W_\beta (T\lambda)} e^{2\e \psi} |u|^2\\
& \le  \int_{\partial\Omega} |u|^2
+ \frac{C}{T\lambda}
\int_{\partial\Omega} m(x, \beta\B) e^{2\e \psi} |u|^2\\
&\le  \int_{\partial\Omega} |u|^2
+ \frac{C}{T} \int_{\partial\Omega} e^{2\e \psi} |u|^2,
\endaligned
$$
where we have used \eqref{L-2-1} for the last inequality.
By choosing $T=(1/2) C$, we obtain 
$$
\int_{\partial\Omega} e^{2\e \psi} |u|^2 \le 2 \int_{\partial\Omega} |u|^2.
$$
In view of part (1) of Lemma \ref{lemma-a2}, we have proved \eqref{L-2}.
The estimates \eqref{L-3}-\eqref{L-3a} follow readily from \eqref{L-2-1} and \eqref{L-2}.
\end{proof}


\section{Pointwise bounds}

Throughout this section we assume that $\Omega$ is a bounded Lipschitz domain.

\begin{lemma}\label{lemma-I}
Suppose $(D+\A)^2 u=0$ in $\b(x_0, r)$. Then
\begin{equation}\label{I-1}
|u(x_0)|\le \fint_{\b(x_0, r)} |u|.
\end{equation}
\end{lemma}

\begin{proof}
This is well-known and follows from the observation  that if $u_\e=\sqrt{|u|^2 +\e^2}$, then
$\Delta u_\e \ge 0$.
\end{proof}

Under the condition that for some $c_0>0$, 
\begin{equation}\label{low-1}
c_0 \int_\Omega | u|^2 
\le \int_\Omega |(D+\A)u|^2 \quad \text{ for any } u \in H^1(\Omega; \C),
\end{equation}
 a Neumann function $N_\A(x, y)$ was constructed in \cite{Shen-2025-2} for the operator 
 $(D+\A)^2$ in $\Omega$,
with the following properties, 
\begin{enumerate}

\item

$N_\A(x, y)$ is H\"older continuous in  the set 
$\big\{(x, y): x, y \in \overline{\Omega} \text{ and } x\neq y\big \}$.

\item

For any $x, y\in \overline{\Omega}$ and $x\neq y$, 
\begin{equation}\label{Ne}
|N_\A(x, y) |
\le \left\{
\aligned
&  C |x-y|^{2-d} & \quad & \text{ if } d\ge 3,\\
&  C_\sigma |x-y|^{-\sigma} & \quad & \text{  if } d=2,
\endaligned
\right.
\end{equation}
where $\sigma\in (0, 1)$. The constants $C, C_\sigma$ depend on $\Omega$ and $c_0$ in \eqref{low-1}.

\item

For any $F\in L^\infty(\Omega; \C)$ and $g\in L^\infty(\partial \Omega; \C)$, the weak solution to the Neumann
problem, 
\begin{equation}\label{NP-2}
(D+\A)^2u = F \quad \text{ in } \Omega \quad
\text{ and } \quad
n \cdot i (D+\A)u=g \quad \text{ on } \partial\Omega, 
\end{equation}
is given by
\begin{equation}\label{sol}
u(x) =\int_\Omega N_\A (x, y) F(y)\, dy
+ \int_{\partial\Omega} N_\A(x, y) g(y)\,  dy
\end{equation}
for any $x\in \overline{\Omega}$.
\end{enumerate}

\begin{remark}\label{re-B}
{\rm

Let $u\in H^1(\Omega; \C)$ be a weak solution of \eqref{NP-2}
with $F\in L^2(\Omega; \C)$ and $g \in L^2(\partial\Omega; \C)$.
Testing \eqref{NP-2}with $u$, or equivalently integrating by parts, and then using \eqref{low-1} and a trace inequality give
\begin{equation}\label{energy}
\| u \|_{H^1(\Omega)}
\le C(\A) \left\{ \| F \|_{L^2(\Omega)} + \| g \|_{L^2(\partial\Omega)} \right\},
\end{equation}
where $C(\A)$ depends on $\A$. By an approximation argument, it follows from \eqref{sol} and \eqref{energy}
that \eqref{sol} continues to hold for a.e.~$x\in \Omega$.
Using  the trace inequality $\| u \|_{L^2(\partial\Omega)}
\le C \| u \|_{H^1(\Omega)}$, one can also show that \eqref{sol} holds
 for a.e.~$x\in \partial\Omega$ (with respect to the surface measure).
}
\end{remark}

\begin{lemma}\label{lemma-B}
Assume that  $\A\in C^1(\overline{\Omega}; \R^d)$ and the condition \eqref{low-1} holds.
Let $x_0\in \partial\Omega$ and $0< r< r_0$.
Suppose  $u\in H^1(\b(x_0, 2r)\cap \Omega; \C)$ and 
\begin{equation}\label{B-1}
\left\{
\aligned
(D+\A)^2 u & =0 &  \quad & \text{ in } \Omega \cap \b(x_0, 2r),\\
n\cdot i (D+\A)u & =\lambda u  &  \quad &  \text{ on } \b(x_0, 2r) \cap \partial\Omega, 
\endaligned\right.
\end{equation}
where $|\lambda| r  \le 1$. Then
\begin{equation}\label{b-1-0}
\sup_{\b(x_0, r)\cap \Omega} |u|
\le C \left\{ \left(\fint_{\b(x_0, 2r)\cap \Omega} |u|^2 \right)^{1/2}
+  \left(\fint_{\b(x_0, 2r)\cap \partial\Omega} |u|^2 \right)^{1/2} \right\},
\end{equation}
where $C$ depends on $\Omega$ and $c_0$ in \eqref{low-1}.
\end{lemma}
\begin{proof}

We assume $d\ge 3$.
The case $d=2$ follows readily from the case $d=3$ by the method of descending. 
First,  observe that by the Caccioppoli inequality,
\begin{equation}\label{Ca}
\int_{\b(x_0, 3r/2)\cap \Omega}
|(D+\A) u|^2
\le \frac{C}{r^2}
\int_{\b(x_0, 2r)\cap \Omega} |u|^2
+ C |\lambda| \int_{\b(x_0, 2r) \cap \partial\Omega} |u|^2.
\end{equation}
Next, let $1< s< t< (3/2)$ and $s_1=(t+s)/2$. 
Choose $\varphi \in C_0^\infty(\b(x_0, tr); \R)$ such that  $ 0\le \varphi\le 1$ in $\b(x_0, tr)$, $\varphi=1$ in $\b(x_0, s_1 r)$,
$|\nabla \varphi |\le C(t-s)^{-1} r^{-1}$ and
$|\nabla^2 \varphi|\le C (t-s)^{-2} r^{-2}$.
Note that  by \eqref{B-1}, 
$$
(D+\A)^2(u\varphi)
= 2 (D+\A)u \cdot D \varphi
+ u D^2 \varphi 
$$
in $\Omega$, and
$$
n\cdot i (D+\A) (u \varphi)=  \lambda u \varphi + i u n\cdot D\varphi
$$
on $\partial\Omega$.
It follows by Remark \ref{re-B} that
\begin{equation}\label{b-1-1}
\aligned
u(x)
 & =\int_\Omega N_\A (x, y) \big\{ 2 (D+\A)u (y) \cdot D \varphi (y) 
+ u (y)  D^2 \varphi (y)\big\} dy\\
 &\qquad
 +  \int_{\partial\Omega}
N_\A(x, y) \big\{  \lambda u (y)  \varphi (y) + i u  (y) n(y) \cdot D\varphi (y)\big \} dy
\endaligned
\end{equation}
for a.e. $x\in \partial\Omega$.
In view of  the estimate \eqref{Ne}, we obtain 
\begin{equation}
\aligned
|u(x)|
 & \le \frac{C_{t, s} }{  r^{d-1}}\int_{\b(x_0, 3r/2)\cap \Omega} |(D+\A) u|
+\frac{C_{t, s} }{ r^{d}} \int_{\b(x_0, 2r)\cap \Omega} |u|\\
& \qquad
+C|\lambda|  \int_{\b(x_0, tr)\cap \partial\Omega} \frac{|u(y)|}{|x-y|^{d-2}} dy
+ \frac{C_{t, s} }{ r^{d-1}} \int_{\b(x_0, 2r)\cap\partial \Omega} |u|
\endaligned
\end{equation}
for a.e. $x\in \b(x_0, s r) \cap \partial\Omega$.
By using the estimates for fractional integrals on $\partial\Omega$, 
this leads to 
\begin{equation}\label{b-1-2}
\aligned
\left(\fint_{\b(x_0, sr)\cap \partial\Omega} |u|^q \right)^{1/q}
 & \le C_{t, s}  r \fint_{\b(x_0, 3r/2) \cap \Omega} |(D+\A) u|
+ C_{t, s} \fint_{\b(x_0, 2r)\cap \Omega} |u|\\
 & \quad
 + C_{t, s} |\lambda| r  \left(\fint_{\b(x_0, tr)\cap \partial\Omega} |u|^p \right)^{1/p}
+  C_{t, s} \fint_{\b(x_0, 2r)\cap \partial\Omega} |u|, 
\endaligned
\end{equation}
where $2\le p< q\le \infty$ and $\frac{1}{p}-\frac{1}{q}< \frac{1}{d-1}$.
Choose a finite sequence $2=p_0<p_1<\cdots < p_N=\infty$ such that
$$
\frac{1}{p_j}-\frac{1}{p_{j+1}} < \frac{1}{d-1}
$$
and radii $1=s_N< s_{N_1} < \cdots < s_0< 3/2$.
Apply \eqref{b-1-2} successively with $(p, q, s, t)
=(p_j, p_{j+1}, s_{j+1}, s_j)$.
H\"older's inequality (equivalently, Jensen's inequality for the normalized averages) 
then gives
\begin{equation}\label{b-1-3}
\sup_{\b(x_0, r)\cap\partial\Omega} |u|
\le
C r \fint_{\b(x_0, 3r/2) \cap \Omega} |(D+\A) u|
+ C \fint_{\b(x_0, 2r)\cap \Omega} |u|
 + C   \left(\fint_{\b(x_0, 2r)\cap \partial\Omega} |u|^2 \right)^{1/2},
\end{equation}
where we have used the assumption $|\lambda| r\le 1$.
The desired estimate \eqref{b-1-0} follows readily from \eqref{Ca} and \eqref{b-1-3}.
\end{proof}

For a function $u$ in $C(\Omega; \C)$, we use $(u)^*$ to denote the nontangential 
maximal function of $u$, defined by
$$
(u)^* (x)
=\sup \{ |u(y)|: \ y \in \Omega \text{ and } |y-x|< C_0  \text{\rm dist}(y, \partial\Omega) \}
$$
for $x\in \partial\Omega$, where $C_0>1$ is a fixed constant depending on the Lipschitz character of $\Omega$.

\begin{lemma}\label{lemma-t}
Let $u\in C^2 (\Omega; \C) \cap C(\overline{\Omega}, \C)$ be a solution of the Dirichlet problem,
\begin{equation}\label{DP-1}
(D+ \A)^2 u=0 \quad \text{ and } \quad
u=f \quad \text{ on } \partial\Omega.
\end{equation}
 Then, for $2\le p\le \infty$, 
\begin{equation}\label{est-t}
\| (u)^*\|_{L^p(\partial\Omega)}
\le C \| f \|_{L^p(\partial\Omega)},
\end{equation}
where $C$ depends on $\Omega$.
\end{lemma}

\begin{proof}
This follows from the observation that $|u|\le v$ in $\Omega$, where $v$ is the harmonic function in $\Omega$ with 
boundary data $|f|$ on $\partial\Omega$.
See \cite[Theorem 2.5]{Shen-2025-2}.
In particular, a solution of \eqref{DP-1} satisfies the maximum principle, 
\begin{equation}\label{max-p}
\max_{\overline{\Omega}} |u| \le \max_{\partial\Omega} |u|.
\end{equation}
\end{proof}

\begin{thm}\label{thm-p0}
Assume that  $\A\in C^1(\overline{\Omega}; \R^d)$ and the condition \eqref{low-1} holds.
Let $u\in H^1(\Omega; \C)$ be a  weak solution of 
\begin{equation}\label{EF-2}
\left\{
\aligned
(D+\A )^2 u  & =0& \quad & \text{ in } \Omega, \\
n \cdot i (D+ \A)u & =\lambda u & \quad & \text{ on } \partial\Omega
\endaligned
\right.
\end{equation}
for some $\lambda>0$. Then $u \in C(\overline{\Omega}; \C)$ and for any $x\in \overline{\Omega}$, 
\begin{equation}\label{p0-1}
|u(x)|\le C \lambda^{\frac{d-1}{2}} \| u \|_{L^2(\partial\Omega)},
\end{equation}
where $C$ depends on $\Omega$ and $c_0$ in \eqref{low-1}.
\end{thm}

\begin{proof}
The fact $u\in C(\overline{\Omega}; \C)$ follows from the representation  formula \eqref{b-1-1}.
To show \eqref{p0-1}, 
we use  \eqref{low-1} and  the trace inequality,
\begin{equation}
c\int_{\partial\Omega} |u|^2
\le \int_\Omega | (D+\A) u|^2 + \int_\Omega |u|^2
\end{equation}
for $u\in H^1(\Omega; \C)$ to obtain  $\lambda\ge c$ for some $c>0$ depending only on $\Omega$ and $c_0$.
Let $r=c \lambda^{-1}< r_0$.
If $\b(x_0, r)\subset \Omega$, we use the interior estimate \eqref{I-1} to obtain 
\begin{equation}
\aligned
|u(x_0)| & \le \left(\fint_{\b(x_0, r)} |u|^2 \right)^{1/2}
\le \frac{C}{r^{\frac{d-1}{2}}} \| (u)^*\|_{L^2(\partial\Omega)}\\
& \le C \lambda^{\frac{d-1}{2}} \| u \|_{L^2(\partial\Omega)},
\endaligned
\end{equation}
where we have used \eqref{est-t} for the last step.
This gives the estimate \eqref{p0-1} for any $x\in \Omega$ with dist$(x, \partial\Omega)\ge c\lambda^{-1}$.

For $x\in \overline{\Omega}$ with dist$(x, \partial\Omega)<  c\lambda^{-1}$.
We choose $x_0\in \partial\Omega$ so that $x\in \b(x_0, c\lambda^{-1})$.
It follows from \eqref{b-1-0} that 
\begin{equation}
\aligned
|u(x)|
&\le C \left\{ { r^{\frac{1-d}{2}}}  \| u \|_{L^2(\partial\Omega)}
+{ r^{\frac{1-d}{2}}} \|  (u)^* \|_{L^2(\partial\Omega)}\right\}\\
& \le C  \lambda^{\frac{d-1}{2}} \| u \|_{L^2(\partial\Omega)},
\endaligned
\end{equation}
where $r=c\lambda^{-1}$ and we have used \eqref{est-t} for the last inequality.
\end{proof}


We are now ready to give the proof of Theorem \ref{main-thm}.

\begin{proof}[\bf Proof of Theorem \ref{main-thm}]

Let $(\e, T, \beta_0, C)$ be given by Theorem \ref{L-2-thm}.
Recall from \eqref{d} that 
$$
d_{\beta, T \lambda} (x) =d_\beta (x, W_\beta (T \lambda)).
$$
Let $u$ be a solution of \eqref{eigen} for some $\beta>\beta_0$ and $\lambda>0$.
Without the loss of generality, we assume $\| u \|_{L^2(\partial\Omega)}=1$.
For $x_0\in \overline{\Omega}$, let
$$
r=r(x_0)=c\,  m(x_0, \beta \B)^{-1}.
$$
\noindent {\bf Case I.}
Suppose $\b(x_0, r(x_0))\subset\Omega$. It follows from \eqref{I-1} that 
\begin{equation}
\aligned
|u(x_0)| & \le \left(\fint_{\b(x_0, r)} |u|^2 \right)^{1/2}\\
& \le C e^{-\e d_{\beta, T\lambda} (x_0)}  m(x_0, \beta \B)^{-1}
\left(\fint_{\b(x_0, r)}  m(x, \beta\B)^2 | e^{\e d_{\beta, T\lambda} (x) } u|^2 \right)^{1/2},
\endaligned
\end{equation}
where we have used the facts that $m(x, \beta \B)\approx m(x_0, \beta \B)$ and
$|d_{\beta, T\lambda}(x)-d_{\beta, T \lambda}(x_0)|\le C$ for $x\in \b(x_0, m(x_0, \beta\B)^{-1})$.
In view of \eqref{L-3}, it follows that 
\begin{equation}\label{p-10}
|u(x_0)|
  \le C e^{-\e d_{\beta, T\lambda} (x_0)}  \big\{ m(x_0, \beta \B) \big\}^{\frac{d-2}{2}}  \lambda^{\frac{1}{2}}.
\end{equation}
This, together with \eqref{a-3}, gives
\begin{equation}
\aligned
|u(x_0)|
   &  \le C  \lambda^{\frac{d-1}{2}} \big\{ 2+ d_{\beta, T\lambda} (x_0) \big\}^{\ell_1 d }
   e^{-\e d_{\beta, T\lambda} (x_0)} \\
   &  \le C  \lambda^{\frac{d-1}{2}}
   e^{-\e_1 d_{\beta, T\lambda} (x_0)},
   \endaligned
\end{equation}
where $\e_1=(\e/2)$.

\medskip

\noindent{\bf Case II.} Suppose $\b(x_0, r(x_0))\cap \partial\Omega\neq \emptyset$.
Choose $y_0\in \b(x_0, r(x_0))\cap \partial\Omega$.
Then $y_0\in \partial\Omega$ and  $x_0\in \b (y_0, C m(y_0, \beta \B)^{-1}) \cap\Omega$.
If $y_0\in W_\beta (T \lambda)$, we have
$$
d_{\beta, T \lambda} (x_0)
\le d_{\beta, T \lambda}(y_0) + d_\beta (x_0, y_0)
\le C.
$$
Hence, by Theorem \ref{thm-p0},
$$
|u(x_0)|\le C \lambda^{\frac{d-1}{2}} \le C e^{-\e d_{\beta, T\lambda}(x_0)}.
$$
If $y_0\notin W_\beta (T \lambda)$, then  $m(y_0, \beta \B) \ge T \lambda$.
Let $r= C m(y_0, \beta \B)^{-1}<  C \lambda^{-1}$.
It follows from \eqref{b-1-0}, which continues to hold if $ \lambda r \le C$,   that 
$$
\aligned
|u(x_0)|^2
 & \le \frac{C}{r^{d-1}}\int_{\b(y_0, 2r)\cap\partial \Omega} |u|^2
+\frac{C}{r^d}\int_{\b(y_0, 2r) \cap \Omega} |u|^2\\
& \le C e^{- 2\e d_{\beta, T \lambda} (x_0)}
\left\{ \frac{1}{r^{d-1}}
\int_{\b (y_0, 2r)\cap \partial\Omega}  e^{2 \e d_{\beta, T\lambda}} |u|^2
+\frac{1}{r^d}
\int_{\b(y_0, 2r)\cap \Omega}
e^{2 \e d_{\beta, T\lambda}} |u|^2\right\}.
\endaligned
$$
 In view of \eqref{L-2}-\eqref{L-3}, this leads to 
 \begin{equation}
 \aligned
 |u(x_0)|^2
  & \le C  e^{- 2\e d_{\beta, T \lambda} (x_0)}
 \left\{  r^{1-d} + \lambda r^{2-d} \right\} \\
 &  \le C  e^{- 2\e d_{\beta, T \lambda} (x_0)}
 m(y_0, \beta \B)^{d-1}\\
   &  \le C  e^{- 2\e d_{\beta, T \lambda} (x_0)}
 m(x_0, \beta \B)^{d-1},
 \endaligned
 \end{equation}
 where we have used $r\lambda< C$.
 This, together with \eqref{a-3}, yields the desired estimate \eqref{main-1} with $\e_1$ in the place of $\e$,
 as in Case I.
\end{proof}


\section{Some examples}

Suppose $\B$ is of finite type on $\overline{\Omega}$ and $\beta > \beta_0$.
Let $u$ be  an eigenfunction associated with the ground state energy $\lambda
=\lambda^{D\!N}(\beta \A, \Omega)\approx \beta^{\frac{1}{\kappa_0+2}}$, 
where $\kappa_0\ge 0$ is given by \eqref{kappa}.

\medskip

\noindent{\bf The non-vanishing case.}

Suppose $\tau_0=\min_{\partial\Omega} |\B| >0$ and thus $\kappa_0=0$.
Then $\lambda\approx \beta^{\frac12}$ for $\beta>\beta_0$.
Also,  $ \widetilde{m}(x, \beta \B)\approx \beta ^{\frac12}$ for $x$ close to $\partial\Omega$ and $\beta>\beta_0$.
Hence, $d_\beta (x, \partial\Omega)\approx \beta^{\frac12} \text{\rm dist} (x, \partial \Omega)$
for $x$ close to $\partial\Omega$.
It follows from \eqref{main-1} that   if $\beta >\beta_0$, 
\begin{equation}\label{nv-1}
| u(x)|\le C \beta^{\frac{d-1}{4}} \exp \big\{ -\e \beta^{\frac12} \text{\rm dist}(x, \partial\Omega) \big\}
\| u \|_{L^2(\partial\Omega)}
\end{equation}
for $x$ close to $ \partial\Omega$,
where $C>1$, $\e>0$ and $\beta_0>1$ depend only on $\Omega$ and $\B$.
By the maximum principle \eqref{max-p},  the estimate \eqref{nv-1}   holds for any $x\in \overline{\Omega}$.

\medskip

\noindent{\bf A case of higher-order vanishing.}

Suppose that $\kappa_0\ge 1$ and the set $\Gamma=\{ x\in \partial\Omega:  \kappa (x)=\kappa_0 \}$ is discrete.
Also assume that  there exists $c>0$ such that 
\begin{equation}\label{v-1}
\sum_{|\alpha|\le \kappa_0-1}
|\partial^\alpha \B(x)|^{\frac{1}{\kappa_0-|\alpha|}}
\ge c\,  \text{\rm dist}(x, \Gamma)
\end{equation}
 for any $x\in \partial\Omega$.
 Recall that  $\lambda\approx \beta^{\frac{1}{\kappa_0+2}}$ for $\beta> \beta_0$.
It follows from \eqref{WW} and \eqref{v-1} that
\begin{equation}\label{W-11}
W_\beta (T \lambda)
\subset \big\{ x\in \partial\Omega: 
\text{\rm dist}(x, \Gamma) \le C \beta^{-\frac{1}{\kappa_0+2}}\big \}.
\end{equation}
Note that by \eqref{mm} and \eqref{v-1},
\begin{equation}
\aligned
d_\beta (x, x_0)
 & \ge c \min_{|\alpha|\le \kappa_0-1} \beta^{\frac{1}{|\alpha|+2}} |x-x_0|^{\frac{\kappa_0+2}{|\alpha|+2}}\\
 &\ge c ( \beta |x-x_0|^{\kappa_0+2})^{\frac{1}{\kappa_0+1}}
\endaligned
\end{equation}
if  $x\in \overline{\Omega} \cap \b(x_0, c_0)$ for some $x_0\in \Gamma$ and $|x-x_0|\ge \beta^{-\frac{1}{\kappa_0+2}}$.
By adjusting the constant $c$, the estimate above continues to hold for any $x\in\overline{\Omega}$ with
dist$(x, \Gamma) \ge \beta^{-\frac{1}{\kappa_0+2}}$.
As a result, we deduce from \eqref{main-1} that  for $\beta$ large, 
\begin{equation}\label{v-2}
|u(x)| \le C\beta^{\frac{d-1}{2 (\kappa_0+2)}}
\exp\big\{ -\e \big [ \beta\,  \text{\rm dist}(x, \Gamma)^{\kappa_0+2}]^{\frac{1}{\kappa_0+1}}\big\}
\| u \|_{L^2(\partial\Omega)}
\end{equation}
for any $x\in \overline{\Omega}$, where $C, \e>0$ are independent of $\beta$.

 \bibliographystyle{amsplain}
 
\bibliography{S2025-3.bbl}

\providecommand{\bysame}{\leavevmode\hbox to3em{\hrulefill}\thinspace}
\providecommand{\MR}{\relax\ifhmode\unskip\space\fi MR }
\providecommand{\MRhref}[2]{%
  \href{http://www.ams.org/mathscinet-getitem?mr=#1}{#2}
}
\providecommand{\href}[2]{#2}
\begin{thebibliography}{10}

\bibitem{CGHP}
Tirumala Chakradhar, Katie Gittins, Georges Habib, and Norbert Peyerimhoff,
  \emph{A note on the magnetic {S}teklov operator on functions}, Mathematika
  \textbf{71} (2025), no.~4, Paper No. e70037, 35.

\bibitem{C-2023}
Bruno Colbois, Alexandre Girouard, Carolyn Gordon, and David Sher, \emph{Some
  recent developments on the {S}teklov eigenvalue problem}, Rev. Mat. Complut.
  \textbf{37} (2024), no.~1, 1--161.

\bibitem{Dauge-2020}
Monique Dauge, Jean-Philippe Miqueu, and Nicolas Raymond, \emph{On the
  semiclassical {L}aplacian with magnetic field having self-intersecting zero
  set}, J. Spectr. Theory \textbf{10} (2020), no.~4, 1211--1252.

\bibitem{Fefferman-1983}
Charles~L. Fefferman, \emph{The uncertainty principle}, Bull. Amer. Math. Soc.
  (N.S.) \textbf{9} (1983), no.~2, 129--206.

\bibitem{Helffer-book}
Soren Fournais and Bernard Helffer, \emph{Spectral methods in surface
  superconductivity}, Progress in Nonlinear Differential Equations and their
  Applications, vol.~77, Birkh\"{a}user Boston, Inc., Boston, MA, 2010.

\bibitem{Helffer-2024}
Bernard Helffer, Ayman Kachmar, and Fran\c{c}ois Nicoleau, \emph{Asymptotics
  for the magnetic {D}irichlet-to-{N}eumann eigenvalues in general domains}, J.
  Funct. Anal. \textbf{290} (2026), no.~12, Paper No. 111449, 48.

\bibitem{Helffer-1996}
Bernard Helffer and Abderemane Mohamed, \emph{Semiclassical analysis for the
  ground state energy of a {S}chr\"{o}dinger operator with magnetic wells}, J.
  Funct. Anal. \textbf{138} (1996), no.~1, 40--81.

\bibitem{Helffer-2001}
Bernard Helffer and Abderemane Morame, \emph{Magnetic bottles in connection
  with superconductivity}, J. Funct. Anal. \textbf{185} (2001), no.~2,
  604--680.

\bibitem{Helffer-F2024}
Bernard Helffer and Fran\c{c}ois Nicoleau, \emph{On the magnetic {D}irichlet to
  {N}eumann operator on the disk: strong diamagnetism and strong magnetic field
  limit}, J. Geom. Anal. \textbf{35} (2025), no.~6, Paper No. 178, 30.

\bibitem{Lu-1999}
Kening Lu and Xing-Bin Pan, \emph{Eigenvalue problems of {G}inzburg-{L}andau
  operator in bounded domains}, J. Math. Phys. \textbf{40} (1999), no.~6,
  2647--2670.

\bibitem{Miqueu-2018}
Jean-Philippe Miqueu, \emph{Eigenstates of the {N}eumann magnetic {L}aplacian
  with vanishing magnetic field}, Ann. Henri Poincar\'{e} \textbf{19} (2018),
  no.~7, 2021--2068.

\bibitem{Montgomery}
Richard Montgomery, \emph{Hearing the zero locus of a magnetic field}, Comm.
  Math. Phys. \textbf{168} (1995), no.~3, 651--675.

\bibitem{Pan-2002}
Xing-Bin Pan and Keng-Huat Kwek, \emph{Schr\"{o}dinger operators with
  non-degenerately vanishing magnetic fields in bounded domains}, Trans. Amer.
  Math. Soc. \textbf{354} (2002), no.~10, 4201--4227.

\bibitem{Raymond-book}
Nicolas Raymond, \emph{Bound states of the magnetic {S}chr\"{o}dinger
  operator}, EMS Tracts in Mathematics, vol.~27, European Mathematical Society
  (EMS), Z\"{u}rich, 2017.

\bibitem{Shen-2025-2}
Zhongwei Shen, \emph{Boundary value problems for magnetic {L}aplacian in
  semiclassical analysis}, Preprint arXiv: 2509.00292.

\bibitem{Shen-2025}
\bysame, \emph{{The magnetic Laplacian with a higher-order vanishing magnetic
  field in a bounded domain}}, Preprint arXiv:2505.03690.

\bibitem{Shen-1995}
\bysame, \emph{{$L^p$} estimates for {S}chr\"{o}dinger operators with certain
  potentials}, Ann. Inst. Fourier (Grenoble) \textbf{45} (1995), no.~2,
  513--546.

\bibitem{Shen-1996}
\bysame, \emph{Eigenvalue asymptotics and exponential decay of eigenfunctions
  for {S}chr\"{o}dinger operators with magnetic fields}, Trans. Amer. Math.
  Soc. \textbf{348} (1996), no.~11, 4465--4488.

\bibitem{Elst}
A.~F.~M. ter Elst and El~Maati Ouhabaz, \emph{The diamagnetic inequality for
  the {D}irichlet-to-{N}eumann operator}, Bull. Lond. Math. Soc. \textbf{54}
  (2022), no.~5, 1978--1997.

\end{thebibliography}




\end{document}